\documentclass[review]{elsarticle}

\usepackage{lineno,hyperref}
\modulolinenumbers[5]
\usepackage{amsmath,amsthm}
\usepackage{amssymb}
\usepackage{enumerate}
\journal{Expositiones Mathematicae}

\newtheorem{theorem}{Theorem}[section]
\newtheorem{corollary}[theorem]{Corollary}
\newtheorem{lemma}[theorem]{Lemma}
\newtheorem{proposition}[theorem]{Proposition}

\theoremstyle{definition}

\def\N{{\mathbb N}}
\def\F{{\mathbb F}}
\def\Z{{\mathbb Z}}

\def\R{{\mathbb R}}

\def\I{\mathbb {I}}

\def\scA{\F\monoid}

\def\scK{\mathcal{K}}
\def\genset{\mathcal{X}}

\def\monoid{\left<\genset\right>}
\def\lbrack{\left[}
\def\rbrack{\right]}

\def\freeLie{\mathcal{L}}
\def\ad{{\rm ad}\ }
\def\freeF{\F\monoid}
\def\Heisen{\mathcal{H}(q)}

\def\nqcombi{\{ n \}_q}

\def\A1{\alpha}
\def\B1{\beta}
\def\g{\gamma}
\def\D1{\delta}

\def\zeta{\zeta}
\def\E1{\eta}
\def\T1{\theta}
\def\Io1{\iota}
\def\L1{\lambda}
\def\rq{\mathcal{R}(q)}
\def\R1{\rho}
\def\T2{\tau}
\def\S1{\sigma}

\def\LieAB{\lbrack A,B\rbrack}

\def\algI{I}
\makeatletter
\def\ps@pprintTitle{%
 \let\@oddhead\@empty
 \let\@evenhead\@empty
 \def\@oddfoot{}%
 \let\@evenfoot\@oddfoot}
\makeatother

\begin{document}

\begin{frontmatter}

\title{An extension of a $q$-deformed Heisenberg algebra\\ and its Lie polynomials}


\author[mymainaddress]{Rafael Reno S. Cantuba\corref{mycorrespondingauthor}}
\cortext[mycorrespondingauthor]{Corresponding author}
\ead{rafael\_cantuba@dlsu.edu.ph}

\author[mymainaddress]{Mark Anthony C. Merciales}
\ead{mark\_anthony\_merciales@dlsu.edu.ph}

\address[mymainaddress]{Mathematics and Statistics Department,
De La Salle University,
Manila, Philippines}

\begin{abstract}
\quad \ Let $\mathbb{F}$ be a  field, and  fix a  $q\in\mathbb{F}$. The $q$-deformed Heisenberg algebra $\mathcal{H}(q)$ is the unital associative algebra over $\mathbb{F}$ with  generators $A$, $B$ and a  relation which asserts that $AB - qBA$ is the multiplicative  identity in  $\mathcal{H}(q)$. We extend $\mathcal{H}(q)$ into an algebra $\mathcal{R}(q)$ defined  by  generators  $A$, $B$  and  a relation  which  asserts that  $AB-qBA$ is central in $\mathcal{R}(q)$. We identify all elements of $\mathcal{R}(q)$ that are Lie polynomials in $A$, $B$. 
\end{abstract}

\begin{keyword}
\texttt{$q$-deformed Heisenberg algebra}\sep \texttt{diamond lemma}\sep \\ \texttt{Lie polynomial}\sep \texttt{central extension}
\MSC[2010] 17B60\sep  16S15\sep 47C99\sep 47L15
\end{keyword}

\end{frontmatter}

%\linenumbers

\section{Introduction}\label{intro}
\quad \ By the $q$-deformed Heisenberg algebras $\Heisen$, we refer to the parametric family of unital associative algebras with two generators $A, B$ that satisfy the commutation relation $AB-qBA =I$ where $I$ is the multiplicative identity in $\Heisen$  \cite{Hel05}.
The $q$-deformed Heisenberg algebras have received much attention with regard to their representations and structure as pointed out in \cite{Hel} since even the case $q=1$, which is also called the undeformed case, is already an object of interest with numerous applications \cite{Hel05}. Moreover, a $q$-deformed Heisenberg algebra is an algebraic formalism of the commutation relation obeyed by the creation and annihilation operators in $q$-oscillators \cite{Bie89,Burb,Gaa54a,Gaa54b,Klim,Mac89}.

We consider an extension of $\Heisen$ denoted by $\rq$ which is an associative algebra with generators $A, B$ and relation which asserts that $AB-qBA$ is central in $\rq$. The manner by which we extend the algebra $\Heisen$ into $\rq$ is similar to that done in \cite{UAW} in which an Askey-Wilson algebra was extended into the universal Askey-Wilson algebra. We investigate the algebraic structure of $\rq$ in terms of some essential associative algebra properties, and characterize the Lie subalgebra of $\rq$ generated  by $A \ \mbox{and} \ B$. 

Studies have also been made about similar Lie polynomial characterization problems \cite{Can, Can2, Can3, Can4, Can5}. The Lie polynomial characterization problem for a $q$-deformed Heisenberg algebra was solved in \cite{Can2}, and this was extended to a bigger class of algebras in \cite{Can5}. This work is a further generalization of the solution to the Lie polynomial characterization problem in \cite{Can5} since $\rq$ represents a bigger class of algebras that includes those considered in \cite{Can5}. 

 We give other presentations for $\rq$ with some interesting properties that are helpful in elucidating further properties of the algebra that relate to reduction systems, normal forms, and Lie polynomials. More specifically, we investigate the structure of $\rq$ by first determining a basis using the \emph{Diamond Lemma for Ring Theory} [1,Theorem 1.2]. We use this basis to further understand the structure of the Lie subalgebra $\freeLie$ of $\rq$ generated  by $A \mbox{ and } B$. 

The algebra $\rq$ is formally defined as an extension of a $q$-deformed Heisenberg algebra in Section~\ref{sec3}, where we also introduce a new generator $\g:=AB -qBA$ resulting to another presentation for $\rq$ with interesting properties.  By introducing a new generator, in addition to the fact that we obtain a relatively better presentation for $\rq$, computations involving its elements become more manageable, and gives a better setting for determining a basis for $\rq$ using the Diamond Lemma. In Section~\ref{sec4}, we give $\rq$ another presentation with four generators after introducing another generator $C:=AB-BA$. Also, we determine a new basis for $\rq$ then investigate some of its properties involving the Lie bracket operation. Lastly, we present in Section~\ref{sec5} the results related to the Lie subalgebra $\freeLie$ of $\rq$ generated by $A$ and $B$, or the set of all Lie polynomials in $A, \, B$. Moreover, taking the quotient of $\rq$ modulo the relation $\g=\algI$ yields the solution to the Lie polynomial characterization problem for $\Heisen$ that was solved in \cite{Can2}, which now is a specific case of our solution for the Lie polynomial characterization problem in $\rq$ that we show in this work. Similarly, the results in \cite{Can5} is the specific case of our solution in this paper if we take the quotient of $\rq$ modulo the relation $\g=b\algI$ for some scalar $b$.
%%%%%%%%%%%%%%%%%%%%%%%%%%%%%%%%%%%%%%%%%%%%%%%%%%%%%%%%%%%%%%%%%%%%%%%%
\section{Preliminaries}\label{prelim}
\quad \ \ Let $\F$ be a field. An \emph{associative algebra} over $\F$, or an \emph{$\F$-algebra}, or simply an \emph{algebra} over $\F$, is a vector space $\mathcal{A}$ over $\F$ together with a bilinear vector multiplication operation $\scA \times\scA \rightarrow \scA$ given by $(U,V)\mapsto UV$ such that $\scA$ is a ring with respect to vector addition and vector multiplication. If an associative algebra $\scA$ has an identity element under vector multiplication, then we say that $\scA$ is \emph{unital}.

Denote the set of all nonnegative integers by $\N$, and the set of all positive integers by $\Z^+$. If $n \in \N$, let $\genset$ denote an $n$-element set. We refer to the elements of $\genset$  as \emph{letters} or \emph{generators}.

Given $t \in \N$, by a \emph{word of length t} on $\genset$ we mean a sequence of the form 
\begin{equation}
    	X_1X_2 \cdots X_t
\end{equation}
where $X_i \in \genset$ for 1 $\le i \le t$. The word of length 0 or the \emph{empty word} is denoted by \textit{I}. Let $\monoid$ be the set of all words on $\genset$. Given $W \in \monoid$, we denote the length of $W$ by $|W|$.
For a word $W$=$X_1X_2 \cdots X_{|W|}$ on $\genset$, a sequence of the form \begin{equation}
    	X_sX_{s+1} \cdots X_t
\end{equation}
where $s,t \in \Z^+$ and $s\le t \le |W|$ is a \emph{subword} of $W$.
Given words $X_1X_2 \cdots X_s$ and $Y_1Y_2 \cdots Y_t$, their \emph{concatenation product} is given by 
\begin{equation}
    X_1X_2 \cdots X_sY_1Y_2 \cdots Y_t.
\end{equation}

 We now recall the \emph{free unital associative algebra} over $\F$ generated by $\genset$ which we denote by $\freeF$. The associative algebra $\freeF$ has basis $\monoid$. Multiplication in $\freeF$ is determined by the concatenation product. 

Denote the elements of $\genset$ by $X_1, \, X_2, \, \ldots, \,  X_n$ and let $L_1,  \, R_1, \, L_2, \, R_2, \, \ldots, \\ \, L_m, \, R_m \in \freeF$. Let $\I$ be the ideal of $\freeF$ generated by $L_1 - R_1, \, L_2 - R_2, \, \ldots, \, L_m - R_m$. 
The associative algebra with generators $X_1  , X_2 , \ldots , X_n$ and relations $L_1 = R_1, L_2 = R_2, \ldots , L_m = R_m$
is the quotient algebra $\freeF/\I$. 

We define a new operation in $\scA$ which is
\begin{eqnarray} \label{relaf2.5}
\nonumber [X, Y]&:=&XY-YX, \quad  \quad \forall \, X, \, Y \in \scA.
\end{eqnarray}
Clearly, the operation $\lbrack\cdot,\cdot\rbrack:\F\monoid\times\F\monoid\rightarrow\F\monoid$ has the property
\begin{eqnarray} \label{relaf2.5}
\nonumber [X, X]&=&0, \quad  \quad \forall \, X\in \scA.
\end{eqnarray}
Also, $\lbrack\cdot,\cdot\rbrack$ is bilinear, and is  \emph{skew-symmetric}, that is, 
\begin{eqnarray}
\nonumber [X,Y]&=&-[Y,X],\quad \ \quad  \quad   \quad \forall \, X, \, Y \in \scA,
\end{eqnarray}
and it also satisfies the \emph{Jacobi identity}
\begin{eqnarray}
\nonumber [X,[Y,Z]]+[Y,[Z,X]]+[Z,[X,Y]] &=&0,\quad \ \quad  \quad   \quad \forall \, X, \, Y\,Z \in \scA.
\end{eqnarray}
By these properties,  $\scA$ is a Lie algebra with $\lbrack\cdot,\cdot\rbrack$ as the Lie bracket.
%%%%%%%%%%%%%%%%%%%%%%%%%%%%%%%%%%%%%%%%%%%%%%%%%%
\section{The algebra $\rq$}\label{sec3}
\quad \ Throughout, we fix a $q \in\F$. 
Let $\genset_1 = \{A, \, B \}$, and denote the multiplicative identity in $\F\monoid$ by $\algI$. Let $\I_1(q)$ be the two-sided ideal of $\F\left<\genset_1\right>$ generated by $AB-qBA-I$. The $q$-deformed Heisenberg algebra or the associative algebra $\Heisen$ with generators $A, B$ and relation $AB-qBA=I$ is the quotient algebra $\F\left<\genset_1\right>/\I_1(q)$.
We extend $\Heisen$ to an associative algebra $\rq$ defined by generators $A,B$ and a relation which asserts that $AB-qBA$ is central in $\rq$. This means that the relations
\begin{eqnarray}
  A(AB-qBA)&=&(AB-qBA)A, \mbox{ and } \label{eq2} \\
  B(AB-qBA)&=&(AB-qBA)B,  \label{eq3}
\end{eqnarray}
are sufficient to define $\rq$.

By introducing a new letter $\g$, the algebra $\rq$ would have the following presentation.
\begin{lemma}\label{lemfi3.1}
Let $\g:=AB-qBA$. The algebra $\rq$ has a presentation by generators $A, \, B, \, \g$ and relations
\begin{eqnarray}
   \g &=& AB-qBA, \label{eq4}  \\ 
    \g A &=& A\g, \label{eq5} \\ 
    \g B &=& B\g. \label{eq6}
\end{eqnarray}
\end{lemma}
\begin{proof}
 Notice that relation \eqref{eq4} is precisely how we defined $\g$ while \eqref{eq5} and \eqref{eq6} follow from the assertion that $AB-qBA$  is central in $\rq$.
 
 In \eqref{eq5} and \eqref{eq6}, we use \eqref{eq4} to recover \eqref{eq2} and \eqref{eq3}. 
\end{proof}

We use the presentation for $\rq$ given in Lemma~\ref{lemfi3.1} above to determine a basis for $\rq$ using the Diamond Lemma. 
 \begin{theorem}
The following elements form a basis for $\rq$.
\begin{eqnarray}\g^{h} B^{m}  A^{n} , \quad (h, \, m, \, n  \in \N). \label{eq7}
\end{eqnarray}
\end{theorem}
\begin{proof}
From relations \eqref{eq4}, \eqref{eq5}, and \eqref{eq6} which define $\rq$, we form the following reduction system $S$.
\begin{eqnarray}
& S= \{\L1=(AB, \g + qBA), \, \S1=(A \g, \g A), \,  \T2=(B\g, \g B)\}.& \label{eq8}
\end{eqnarray}
With the reduction system $S$, the only ambiguity is $(\L1 , \T2 , A, B,  \g)$ which involves the word $AB \g$. It can be easily shown that $r= r_{I\T2 A} \circ r_{B\S1 I}$ and $r' = r_{\g \L1 I}\circ r_{I\S1 B}$ are the desired composition of reductions such that
\begin{eqnarray}
r(f_{\L1} \g)=r'(Af_\T2). \label{eq9}
\end{eqnarray}
Hence the only ambiguity $(\L1 , \T2 , A, B,  \g)$ of $S$ is resolvable. By the Diamond Lemma, the set of all irreducible words on $\genset$ with respect to the reduction system $S$ forms a basis for $\rq$. What remains to be shown is that this basis is equal to the set 
\begin{eqnarray}
   \{\g ^{h} B^{m}  A^{n} :\ \ h, \, m, \, n \in\N \}. \label{eq10}
\end{eqnarray}
It is clear that elements in \eqref{eq10} are irreducible with respect to the reduction system $S$.

Now suppose we have a word $W \notin \{ \g^{h} B^{m}  A^{n} :\ \ h, \, m, \, n \in\N \}$. Then $W$ has a subword of the form $A^{j}B^{k}\g^{l}$  where $j, \, k, \, l \in \N$ and at most one of these powers is equal to zero. We have these cases to consider: when  $j=0,  \, k=0, \, l=0$, and if $j, \, k, \, l \in \Z^+$.  If $j=0$, then a reduction which involves $\T2$ would act nontrivially on $A^{j}B^{k}\g^{l} = B^{k}\g^{l}$. For $k=0$,  a reduction which involves $\S1$ would act nontrivially on $A^{j}B^{k}\g^{l} = A^{j}\g^{l}$, and if $l=0$, a reduction which involves $\L1$ would act nontrivially on $A^{j}B^{k}\g^{l} = A^{j}B^{k}$. For the case $j, \, k, \, l$ $ \in\Z^+$,  a reduction which involves $\L1$ or $\T2$ would act nontrivially on $A^{j}B^{k}\g^{l}$.

In any of these cases, $W$ is not irreducible. Hence, any element in \eqref{eq10} is irreducible with respect to the reduction system $S$. 
The result follows. 
\end{proof} 
%%%%%%%%%%%%%%%%%%%%%%%%%%%%%%%%%%%%%%%%%%%%

We recall the following relations on $q$-special combinatorics from [7, Appendix C]. For a given $n \in \N$, $p \in \Z$ and $z \in \F$,
\begin{eqnarray}
\{n\}_z &:=& \sum_{l=0}^{n-1}z^l, \label{eq11} \\ 
\{n\}_z! &:=&\prod_{l=1}^{n}\{l\}_z, \label{eq12}\\
\binom{n}{p}_z &:=& \frac{\{n\}_z !}{\{p\}_z !\{n-p\}_z !}. \label{eq13} 
\end{eqnarray}
If $n \leq 0$, then we interpret \eqref{eq11} as the empty sum 0, and \eqref{eq12} as the empty product 1.
If  $p<0$ or $p>n$,  we interpret \eqref{eq13} to be zero and is equal to 1 if $p=0$ or $p=n$. Additionally, we also have the following relation.
\begin{eqnarray}
\binom{n}{p-1}_z +z^p\binom{n}{p}_z &=& \binom{n+1}{p}_z =z^{n+1-p}\binom{n}{p-1}_z +\binom{n}{p}_z.  \label{eq272}
\end{eqnarray}
Now, from \eqref{eq4}, we can derive 
\begin{eqnarray}
    AB=\g + qBA,   \label{eq14}
\end{eqnarray}
the right-hand side of which is clearly a linear combination of elements in \eqref{eq10}. We use \eqref{eq14} in generalizing the formula for the expansion of words of the form $A^j B^k$ where $j, \, k$ $ \in\N$ as linear combinations of basis elements in \eqref{eq7}.
\begin{proposition}
For any $n \in\N$,
\begin{eqnarray}
    A^n B &=& \nqcombi \g A^{n-1} + q^n B A^n , \ \emph{and} \label{eq15} \\
    AB^n &=& \nqcombi \g B^{n-1} + q^n B^n A. \label{eq16}
\end{eqnarray}
\end{proposition}

\begin{proof}
 We shall use induction on $n$. The case $n=0$ is trivial. If $n=1$, the case will simply be reduced to \eqref{eq14} which clearly satisfies \eqref{eq15}. Suppose \eqref{eq15} holds for some $k \in\N$.

Observe that
\begin{eqnarray}
\nonumber A^{k+1}B &=& A(A^k B), \\
\nonumber &=& A(\{k\}_q \g A^{k-1} + q^k B A^k), \\
\nonumber &=& \{k\}_q \g A^{k} + q^k AB A^k, \\
\nonumber &=& \{k\}_q \g A^{k} + q^k(\g + qBA)A^k, \\
\nonumber &=& \{k+1\}_q \g A^{k} + q^{k+1} B A^{k+1}.
\end{eqnarray}
This implies that \eqref{eq15} holds for $k+1$ so by induction, we get the desired result.

By similar routine computations and arguments, the relation \eqref{eq16} can be shown to be true for any $n\in\N$. 
\end{proof}

\begin{lemma} For any $n \in\N$,
\begin{eqnarray} 
    A^n B^n (q^{n+1}BA + \{n+1\}_q \g) &=& A^{n+1} B^{n+1},  \ \emph{and}\label{eq17} \\
 B^n A^n (BA - \nqcombi\g) &=& q^n B^{n+1}A^{n+1}. \label{eq18}  
\end{eqnarray}
\end{lemma}
\begin{proof}
From \eqref{eq16}, we have 
\begin{eqnarray}
AB^n - \nqcombi \g B^{n-1} &=& q^n B^n A.   \label{eq19} 
\end{eqnarray}
Now, observe that 
\begin{eqnarray}
\nonumber A^n B^n (q^{n+1}BA + \{n+1\}_q \g) &=& A^n q^{n+1} B^{n+1}A + \{n+1\}_q \g A^n B^n, \\
\nonumber &=& A^n (AB^{n+1} -\{n+1\}_q \g B^n) + \{n+1\}_q \g A^n B^n, \\
\nonumber &=& A^{n+1} B^{n+1},
\end{eqnarray}
and
\begin{eqnarray}
\nonumber B^n A^n (BA - \nqcombi\g) &=& B^n(A^n BA - \nqcombi\g A^n), \\
\nonumber &=& B^n((\nqcombi \g A^{n-1} + q^n B A^n)A - \nqcombi\g A^n), \\
\nonumber &=& B^n(\nqcombi \g A^n + q^n B A^{n+1} - \nqcombi\g A^n),  \\
\nonumber &=& B^n(q^n BA^{n + 1}), \\
\nonumber &=& q^n B^{n+1} A^{n+1}. \qedhere
\end{eqnarray}
\end{proof}

Observe that if we substitute the corresponding expression for $A^n B^n$ on the left-hand side of \eqref{eq17} and similarly for $A^{n-1}B^{n-1}$ to the result and so on until it is possible, and applying the same process on the left-hand side of \eqref{eq18}, the result is a  product of linear combinations of $BA \mbox{ and } \g$. This repeated process of substitution would result to the following.
\begin{corollary} For any $n \in\N$,
\begin{eqnarray}
    A^n B^n &=& \prod_{i=1}^{n} (q^i BA + \{i\}_q \g ),  \ \emph{and}  \label{eq20} \\
    q^{\binom{n}{2}}B^n A^n &=& \prod_{j=0}^{n-1} (BA - \{j\}_q \g ). \label{eq21}
\end{eqnarray}
\end{corollary}
\begin{proof}
We prove \eqref{eq20} and \eqref{eq21} by induction on $n$. We first consider \eqref{eq20} which holds for $n=1$ since
\begin{eqnarray}
 \nonumber AB&=& \g + qBA, \\
 \nonumber  &=& qBA + \{1\}_q \g, \\
 \nonumber  &=&\prod_{i=1}^{1} (q^i BA + \{i\}_q \g ).
\end{eqnarray}

We assume that \eqref{eq20} holds for some $k \in\N$. Observe that 
\begin{eqnarray}
\nonumber    A^{k+1} B^{k+1} &=& A^k B^k (q^{k+1}BA + \{k+1\}_q \g), \\
\nonumber    &=& \prod_{i=1}^{k} (q^i BA + \{i\}_q \g )(q^{k+1}BA + \{k+1\}_q \g),\\
\nonumber &=& \prod_{i=1}^{k+1} (q^i BA + \{i\}_q \g ).
\end{eqnarray}
This suggests that \eqref{eq20} also holds  for $k+1$ so by induction, \eqref{eq20} is true for any $n \in\N$.

We can also prove \eqref{eq21} by induction on $n$ in similar manner.
\end{proof}
%%%%%%%%%%%%%%%%%%%%%%%%%%%%%%%%%%%%%%%%%%%%%%%%%%%%%%%%%%
\section{Another presentation of $\rq$}\label{sec4}
In this section, we present properties of the element $[A,B]$ of $\rq$. These properties are useful in deriving reordering formulae in $\rq$. 

\begin{lemma} \label{Lemf4.1}
The following relations hold in $\rq$:
\begin{eqnarray}
    A\LieAB &=& q\LieAB A,  \quad \emph{and}\label{eq23}\\
    \LieAB B &=& qB\LieAB   \label{eq24}
    \end{eqnarray}
\end{lemma}
 \begin{proof}
 Note that $A\LieAB = A(AB) - A(BA).$ Replacing $AB$ by the left hand side of \eqref{eq14}, we have $A(\g + qBA) - (\g + qBA)A = qABA - qBAA= q\LieAB A.$ Also, $\LieAB B =(AB)B - (BA)B.$ Using \eqref{eq14} and simplifying the result, we have $qBAB - qBBA = qB\LieAB$. 
 \end{proof}
\begin{proposition}\label{Profin4.2}
For all $ k , \, n \in\N$,
\begin{eqnarray}
    A^k\LieAB &=& q^k \LieAB A^k,  \quad \emph{and}   \label{eq25}\\
    \LieAB B^k &=& q^k B^k \LieAB \label{eq26}
    \end{eqnarray}
    Moreover,
\begin{eqnarray}
    A^k\LieAB^n &=& q^{kn} \LieAB ^n A^k,  \quad \emph{and}  \label{eq27}\\
    \LieAB^n B^k &=& q^{kn} B^k \LieAB^n.   \label{eq28}
    \end{eqnarray}
\end{proposition}
\begin{proof}
We prove \eqref{eq25} and \eqref{eq26} by induction on $k$ while \eqref{eq27} and \eqref{eq28} by induction on $n$.
Notice that if $k=1$, then \eqref{eq25} and \eqref{eq26} would clearly hold by Lemma~\ref{Lemf4.1}.

Suppose \eqref{eq25} and \eqref{eq26} hold for some $r \in \N$.
Observe, 
\begin{eqnarray}
 \nonumber A^{r+1}\LieAB &=& AA^r \LieAB, \\
 \nonumber   &=& Aq^r \LieAB A^r, \\
\nonumber    &=& q^{r+1} \LieAB A^{r+1}.
    \end{eqnarray}
Also, 
\begin{eqnarray}
\nonumber \LieAB B^{r+1} &=&  \LieAB B^rB , \\
\nonumber &=& q^rB^r \LieAB B , \\
\nonumber &=& q^{r+1}B^{r+1} \LieAB. 
\end{eqnarray}
These imply that \eqref{eq25} and \eqref{eq26} also hold for $r+1$ and by induction, both equations also hold for any natural number $k$.

We then prove \eqref{eq27}and \eqref{eq28}. If $n=1$, the case would simply be reduced to the case of \eqref{eq25}  and \eqref{eq26}.

Assume \eqref{eq27} and \eqref{eq28} to be true for some $j \in\N$. Observe that
\begin{eqnarray}
\nonumber  A^k \LieAB ^{j+1} &=& A^k \LieAB^j \LieAB,\\
\nonumber&=& q^{kj}\LieAB^j A^k \LieAB, \\
\nonumber &=&q^{k(j+1)}\LieAB^{j+1} A^k.
\end{eqnarray}

Also, 
\begin{eqnarray}
\nonumber \LieAB^{j+1}B^k &=& \LieAB \LieAB^jB^k, \\
\nonumber &=& q^{jk}\LieAB B^k \LieAB^j, \\
\nonumber &=& q^{k(j+1)}B^k \LieAB^{j+1}.
\end{eqnarray}
These imply that \eqref{eq27} and \eqref{eq28} is also true for $k+1$ and consequently for any natural number $n$ by induction.
\end{proof}

By the Lie backet in $\rq$,  we have 
\begin{eqnarray}
    AB &=& \LieAB + BA.  \label{eq29}
\end{eqnarray}   
We replace $AB$ on \eqref{eq4} by the right-hand side expression of \eqref{eq29} which would give us
\begin{eqnarray}
(1-q)BA &=& \g -\LieAB.   \label{eq30}
\end{eqnarray}

Multiplying both sides of \eqref{eq20} and \eqref{eq21} by $(1-q)^n$, and using \eqref{eq30} to simplify the result, we can derive the equations
\begin{eqnarray}
 (1-q)^nA^n B^n &=& \prod_{i=1}^{n} (\g-q^i\LieAB),   \quad \mbox{and} \label{eq31}  \\
 q^{\binom{n}{2}}(1-q)^nB^n A^n &=&  \prod_{j=0}^{n-1}(q^j\g-\LieAB).   \label{eq32}
\end{eqnarray}
We now expand the products in right-hand sides of equations \eqref{eq31} and \eqref{eq32}.
\begin{theorem} \label{Theo4.3}
For any $n \in \N$,
\begin{eqnarray}
   (1-q)^nA^n B^n  &=& \sum_{i=0}^{n}(-1)^i q^{\binom{i+1}{2}} {\binom{n}{i}}_q \g^{n-i} \LieAB^i, \quad \emph{and}    \label{eq33}\\
    q^{\binom{n}{2}}(1-q)^nB^n A^n &=& \sum_{i=0}^{n}(-1)^i q^{\binom{n-i}{2}} {\binom{n}{i}}_q \g^{n-i}\LieAB^i.\label{eq34}
\end{eqnarray}
\end{theorem}
\begin{proof}
We prove \eqref{eq33} and \eqref{eq34} by induction on $n$. Notice that if $n=0$, \eqref{eq33} and \eqref{eq34} are trivially satisfied.  
Now, we first prove \eqref{eq33}. 

Suppose \eqref{eq33}  holds for some $k \in \N$. We show that it also holds for $k+1$.
\\Observe that,
\begin{eqnarray}
\nonumber (1-q)^{k+1}A^{k+1} B^{k+1}&=&(1-q)^{k+1}AA^{k} B^{k}B.
\end{eqnarray}
Replacing $A^{k} B^{k}$ by its equivalent expression from the assumption then simplifying the result using Proposition~\ref{Profin4.2} would give us
\begin{eqnarray}
 & & (1-q)\left(\sum_{i=0}^{k}(-1)^i q^{\binom{i +1}{2}}q^i {\binom{k}{i}}_q \g^{k-i} \LieAB^i \right)AB. \label{sum1}
\end{eqnarray}
From equation \eqref{eq31}, $(1-q)AB= \g -q[A,B]$, and so \eqref{sum1} becomes
\begin{eqnarray}
 \sum_{i=0}^{k}(-1)^i q^{\binom{i +1}{2}}q^i {\binom{k}{i}}_q \g^{k-i+1} \LieAB^i -\sum_{i=0}^{k}(-1)^i q^{\binom{i +1}{2}}q^{i+1} {\binom{k}{i}}_q \g^{k-i} \LieAB^{i+1} , 
\end{eqnarray}
which is equivalent to 
\begin{eqnarray}
\sum_{i=0}^{k}(-1)^i q^{\binom{i +1}{2}}q^i {\binom{k}{i}}_q \g^{k-i+1} \LieAB^i +\sum_{i=1}^{k+1}(-1)^{i} q^{\binom{i +1}{2}} {\binom{k}{i-1}}_q \g^{k-i+1} \LieAB^{i} \label{sum2}.
\end{eqnarray}
Denote \eqref{sum2} by $L_1$. In order to use relation \eqref{eq272}, we decompose $L_1$ in the following manner
\begin{eqnarray}
\nonumber L_1 &=& (-1)^0 q^{\binom{0 +1}{2}}q^0 {\binom{k}{0}}_q \g^{k-0+1} \LieAB^0 + \sum_{i=1}^{k}(-1)^i q^{\binom{i +1}{2}}q^i {\binom{k}{i}}_q \g^{k-i+1} \LieAB^i \\ 
\nonumber &\quad& +\sum_{i=1}^{k}(-1)^{i} q^{\binom{i +1}{2}} {\binom{k}{i-1}}_q \g^{k-i+1} \LieAB^{i} \\
\nonumber &\quad& +(-1)^{k+1} q^{\binom{(k+1) +1}{2}} {\binom{k}{k}}_q \g^{k-(k+1)+1} \LieAB^{k+1}
\end{eqnarray}
and so 
\begin{eqnarray}
\nonumber L_1 &=& (-1)^0 q^{\binom{0 +1}{2}} {\binom{k+1}{0}}_q \g^{k-0+1} \LieAB^0  \\
\nonumber &\quad& +\sum_{i=1}^{k}(-1)^{i} q^{\binom{i +1}{2}} \left(q^i {\binom{k}{i}}_q +{\binom{k}{i-1}}_q \right) \g^{k-i+1} \LieAB^{i} \\
\nonumber &\quad& +(-1)^{k+1} q^{\binom{(k+1) +1}{2}} {\binom{k+1}{k+1}}_q \g^{k-(k+1)+1} \LieAB^{k+1}, \\
\nonumber &=&\sum_{i=0}^{k+1}(-1)^{i} q^{\binom{i +1}{2}} \ {\binom{k+1}{i}}_q  \g^{k-i+1} \LieAB^{i}
\end{eqnarray}
from which the desired result follows.

Similarly, we can show that if we assume \eqref{eq34} holds for some $j \in \N$, then the same is  true for $j+1$.

Thus, the desired result follows.
  \end{proof}
From here on, let $C:=\LieAB = AB-BA$. We now give another presentation for $\rq$ that involves $C$.
\begin{proposition} \label{theof4.4}
 The algebra $\rq$ has a presentation by generators $A, \, B, \, C, \, \g$ and relations
\begin{eqnarray}
(1-q)AB &=& \g - qC,    \label{eq35}\\
(1-q)BA &=& \g - C,     \label{eq36}\\
AC &=& qCA,                  \label{eq37}\\   
CB &=& qBC,                  \label{eq38}\\
A\g &=& \g A,               \label{eq39}\\
B\g &=& \g B,             \label{eq40}\\
C\g &=& \g C.             \label{eq41}
\end{eqnarray}
\end{proposition}
\begin{proof}
We obtain equation \eqref{eq35} using \eqref{eq14} and $\lbrack A,B\rbrack = AB - BA$. 
 For equation \eqref{eq36}, we only simplify the equivalent expression of the product $qC$ using \eqref{eq14}.
 Equations \eqref{eq37} and \eqref{eq38} follow from Lemma~\ref{Lemf4.1}. Relations \eqref{eq39}, \eqref{eq40} and \eqref{eq41} follow from the assertion that $\g$ is central in $\rq$.
 
 Moreover, using the fact that we defined $C:=\LieAB = AB-BA$, we can recover \eqref{eq4} after replacing $C$ by $AB-BA$ in either \eqref{eq35} or \eqref{eq36}. Also if we replace $C$ by $\LieAB$ in \eqref{eq37} and \eqref{eq38}, the two equations will still turn out to be relations satisfied by $A,B$. See Lemma~\ref{Lemf4.1}. Furthermore, \eqref{eq39} to  \eqref{eq41} would imply that $AB-qBA$ is central. 
\end{proof}

We use the presentation for $\rq$ given in Proposition~\ref{theof4.4} to determine a basis for $\rq$ in four generators $A, \, B, \, C,\, \g$ using the Diamond Lemma.
\begin{theorem}
If $q \neq 0$ and $q \neq 1$, then the vectors
\begin{equation}
    \g ^ h C ^k B^l  , \, \, \g ^h C^k A^t, \quad (h, \, k, \, l \in\N; \ t \in \Z^+) ,    \label{eq44}
\end{equation}
form a basis for $\rq$.
\end{theorem}
\begin{proof}
 We also invoke Bergman's Diamond Lemma to prove this theorem. Using the defining relations of $\rq$ stated in Proposition~\ref{theof4.4}, we can form a reduction system $R$ consisting precisely of the following elements: 
\begin{eqnarray}
\S1_1 &=& \bigg(AB,\frac{\g - qC}{1-q} \bigg), \label{eq46}\\
\S1_2 &=& \bigg(BA, \frac{\g - C}{1-q}\bigg), \label{eq45}\\
\S1_3 &=& (AC,qCA), \label{eq47}\\
\S1_4 &=& (BC,\frac{CB}{q}), \label{eq48}\\
\S1_5 &=& (A\g, \g A), \label{eq49}\\
\S1_6 &=&(B\g, \g B), \label{eq50}\\
\S1_7 &=&(C\g, \g C). \label{eq51}
\end{eqnarray}

With the reduction system $R$, there is no inclusion ambiguity and that the words 
\begin{equation}
    ABA, \ ABC, \ AB\g, \ BAB, \ BAC, \ BA\g, \ AC\g, \ \mbox{and} \ BC\g \label{eq52}
\end{equation}
are precisely the nontrivial words that determine all the overlap ambiguities of the reduction system $R$.
In the following, we show the compositions of reductions that prove that all the ambiguities determined by \eqref{eq52} are resolvable.
\begin{eqnarray}
f_{\S1_1}A &=&\frac{\g A- qCA}{1-q}= r_{I\S1_3I} \circ r_{I\S1_5I}(Af_{\S1_2}), \label{relf4.34}\\
f_{\S1_1}C &=&\frac{\g C- qC^2}{1-q} = r_{I\S1_7I} \circ r_{C\S1_1 I} \circ r_{I\S1_3B}(A f_{\S1_4}),\label{idred} \\
r_{I\S1_7I} (f_{\S1_1} \g)&=&\frac{\g^2 - q\g C}{1-q}=r_{\g \S1_1I} \circ r_{I\S1_5B}(A f_{\S1_6}), \\ 
r_{I\S1_4I} (f_{\S1_2} B) &=& \frac{\g B- qBC}{1-q} =
r_{I\S1_4I} \circ r_{I\S1_6I}( B f_{\S1_1}),\\
f_{\S1_2} C&=&\frac{\g  C - C^2}{1-q}= 
r_{I \S1_7I} \circ r_{C\S1_2I} \circ r_{I\S1_4A}(B f_{\S1_3}),\label{idred2} \\
 r_{I\S1_7I}(f_{\S1_2} \g)&=& \frac{\g^2 -\g C}{1-q} =
r_{ \g \S1_2 I} \circ r_{I \S1_6A} ( B f_{\S1_5}), \\
 r_{I\S1_7A}\circ r_{C \S1_5I}(f_{\S1_3} \g)  &=& q\g C A=
r_{\g \S1_3 I} \circ r_{I\S1_5C}(Af_{\S1_7}), \\ 
r_{I\S1_7B} \circ r_{C\S1_6I}(f_{\S1_4} \g) &=&\frac{ \g CB}{q}=
r_{\g \S1_4I} \circ r_{I \S1_6C}(Bf_{\S1_7}).
\label{relf4.41}
\end{eqnarray}
The relations 
\eqref{relf4.34} to \eqref{relf4.41}
can be proven by routine computations. For the relations \eqref{relf4.34}, \eqref{idred}, \eqref{idred2}, the composition of reductions on the left-hand side can be taken as the identity linear map, which can be interpreted as an empty composition of reductions with respect to the reduction system $R$.
Hence, by the Diamond Lemma, the set of all irreducible words generated by $A, B, C, \g$ with respect to the reduction system $R$ forms a basis for $\rq$. We show that this basis is equal to the set 
\begin{equation}
    \{ \g ^ h C^k B^l ,\, \, \g ^h C^k A^t :\ \ h, \, k, \, l \in\N; \ t \in \Z^+ \}.    \label{eq53}
\end{equation}

Base on respective first components of each ordered pair in $R$, it  is clear that elements in \eqref{eq53} are irreducible with respect to the reduction system $R$. 

Now, suppose $W$ is not in \eqref{eq53}. Then $W$ has a subword of the form $B^x C^y  \g^z$ or $A^x C^y \g^z$  where $x,\, y,\, z \in\N$ and at most one of these powers is equal to zero. We have these cases to consider: when  $x=0,\, y=0, \, z=0$, and if $x,\, y,\, z \in \Z^+$. If $x=0$, then a reduction which involves $\S1_7$ would act nontrivially on  $B^x C^y  \g^z = C^y \g^z$,  while a reduction which involves $\S1_7$ would act nontrivially on  $A^x C^y \g^z =C^y \g^z$. For $y=0$, a reduction which involves $\S1_6$ would act nontrivially on  $B^x C^y  \g^z = B^x \g^z$,  while a reduction which involves $\S1_5$ would act nontrivially on  $A^x C^y \g^z =A^x \g^z$. And if $z=0$, a reduction which involves $\S1_4$ would act nontrivially on  $B^x C^y  \g^z = B^x C^y$, while a reduction which involves $\S1_3$ would act nontrivially on  $A^x C^y \g^z =A^x C^y$.
For the case $x, y, z \in\Z^+$, reductions which involves $\S1_4$ and  $\S1_7$ would act nontrivially on  $B^x C^y  \g^z $,  while reductions which involves $\S1_3$ and  $\S1_7$ would act nontrivially on  $A^x C^y \g^z$.

It is clear that in any of these cases, $W$ is not irreducible. Hence, any element in \eqref{eq53} is irreducible with respect to the reduction system $R$. Thus, the result follows.
\end{proof}
%%%%%%%%%%%%%%%%%%%%%%%%%%%%%%%%%%%%%%%%%%%%%%%%%%%%%%%%%%%%%%%%%%%%%%%%%%%%%%%
\section{Lie polynomials in $\rq$}\label{sec5}

From this point onward, we assume that $q$ is nonzero and is not a root of unity. Let $\freeLie$ denote the Lie subalgebra of $\rq$ generated by $A \mbox{ and }B$.
 Fix an element $X \in \freeLie$. We recall the linear map $\ad X$ that sends $Y \mapsto [X,Y]$ for any $Y \in \freeLie$. Also, in addition to relations on $q$-special combinatorics from [7, Appendix C] which were  enumerated on  section 3, we will also use the following:

 For a given $n, \, r, \, k \in \N$, and $z \in \F$,
 \begin{eqnarray}
     (1-z)\{n\}_z &=& 1-z^{n}, \label{eq54}\\
       \{n+k\}_z &=& z^k \{n\}_z + \{k\}_z. \label{eq55}
 \end{eqnarray}
In addition,
\begin{eqnarray}
\{rn\}_z &=& \{n\}_z\{r\}_{z^n},  \label{eqf12} 
\end{eqnarray}
\begin{eqnarray} \label{eq271} 
\{n\}_{z^k} &=& \sum_{l=0}^{n-1} (z^k)^l.
\end{eqnarray}

%%%%%%%%%%%%%%%%%%%%%%%%%%%%%

Our main goal for this section is to determine a basis for $\freeLie$. To accomplish the said goal, we first exhibit some important elements of $\freeLie$.

\bigskip
\begin{lemma} \label{lemmaf5.1}
For any $n \in \N$, 
\begin{eqnarray}
((-\ad A)^n) (C)&=& (1-q)^n CA^n,   \label{eq56}\\
((\ad B)^n) (C) &=& (1-q)^n B^n C,    \label{eq57}\\
((-\ad C)^n) (B)&=& (1-q)^n BC^n, \quad \emph{and}   \label{eq58}\\
 ((\ad C)^n) (A)&=&(1-q)^n C^n A.    \label{eq59} 
\end{eqnarray}
\end{lemma} 
\begin{proof}
 We prove equations \eqref{eq56} to \eqref{eq59} by mathematical induction. It can be easily shown that equations \eqref{eq56} to \eqref{eq59} hold for $n=0$.
 
 Suppose \eqref{eq56} and \eqref{eq57} hold for some $j \in \N$. By some routine computations, we can easily obtain 
 \begin{eqnarray}
 ((-\ad A)^{j+1})(C)&=& (1-q)^{j+1}CA^{j+1} \quad \mbox{and}\\
 ((\ad B)^{k+1})(C)&=& (1-q)^{k+1}B^{k+1}C.
 \end{eqnarray}

These suggest that \eqref{eq56} and \eqref{eq57} hold for any natural number $n$.

The proof of \eqref{eq58} is similar to the proof for \eqref{eq56} while the proof of \eqref{eq59} is similar to the proof for \eqref{eq57}. 
\end{proof}

\begin{proposition} \label{Propo5.2}
For any $n, \, m \in \N$
\begin{eqnarray}
((\ad C)^m)(CA^n)&=& (1-q^n)^m C^{m+1} A^n, \quad \emph{and} \label{eq60}\\
((-\ad C)^m)(B^n C)&=& (1-q^n)^m B^n C^{m+1}.  \label{eq61}
\end{eqnarray}
Moreover, the following are elements of $\freeLie$:
  \begin{equation}
  A, \ \ B,\ \ C, \ \ C^{m+1}A^n , \quad B^n C^{m+1} \quad \  (n, \, m \in \N).       \label{eq62}
 \end{equation}  
\end{proposition}

\begin{proof}
 We first prove \eqref{eq60} and \eqref{eq61} by induction on $m$. It is routine to show that equations \eqref{eq60} and \eqref{eq61} for  hold for $m=0$.
 
Now we assume that \eqref{eq60} and \eqref{eq61} are true for some $k \in \N$.
Using the assumption, we can easily derive
\begin{eqnarray}
\nonumber ((\ad C)^{k+1})(CA^n)&=& (1-q^n)^{k+1}C^{k+2} A^n \quad \mbox{and} \\
\nonumber ((-\ad C)^{k+1})(B^n C)&=& (1-q^n)^{k+1}B^n C^{k+2}.
\end{eqnarray}
These imply that \eqref{eq60} and \eqref{eq61} hold for any natural number $m$ . 

Further, the Lie subalgebra  $\freeLie$ is closed under lie bracket operation and has generators $A,\ B$. This would clearly imply that  $A, \, B, \, C \in \freeLie$.

Moreover, the claim that $C^{m+1}A^n, \, B^n C^{m+1} \in \freeLie$ where $m,n \in \N$ follows from \eqref{eq60} and \eqref{eq61}.
\end{proof}

\begin{proposition} \label{prop5.3}
For any $n \in \N$
\begin{equation}
  q^n ( (\ad B) \circ (\ad C)^n ) (A) =   (1-q)^{n}\{n\}_q \g C^n -(1-q)^{n}\{n+1\}_q C^{n+1}
. \label{eq63}
\end{equation}
\end{proposition}
 \begin{proof}
  We use induction on $n$. Equation \eqref{eq63} is trivially satisfied when $n=0$.
  
 We suppose that for some $k \in \N$, \eqref{eq63} holds.
  Observe that 
  \begin{eqnarray}
  \nonumber q^{k+1} ( (\ad B) \circ (\ad C)^{k+1}) (A) &=&  q^{k+1}( (\ad B)((\ad C)^{k+1})(A)), \\
  \nonumber &=& q^{k+1}(\ad B)((1-q)^{k+1}C^{k+1} A), \\
  \nonumber &=& q^{k+1}(1-q)^{k+1}[B,C^{k+1}A], \\
  \nonumber &=& (1-q)^{k+1}(q^{k+1} BC^{k+1}A - q^{k+1} C^{k+1}AB), \\
  \nonumber &=& (1-q)^{k+1}(C^{k+1}BA - q^{k+1} C^{k+1}AB), \\
  \nonumber &=& (1-q)^k(C^{k+1}(1-q)BA - q^{k+1} C^{k+1}(1-q)AB).
  \end{eqnarray}
  
  We replace $(1-q)AB$ and $(1-q)BA$ by the right-hand side expression of \eqref{eq35} and \eqref{eq36} respectively then simplify the result using \eqref{eq54}:
  \begin{eqnarray}
  \nonumber q^{k+1}( (\ad B) \circ (\ad C)^{k+1}) (A)&=& (1-q)^{k+1}\left(\{k+1\}_q \g C^{k+1} - \{k+2\} C^{k+2}\right),   \label{eq65}
  \end{eqnarray}
 and the result follows.
 \end{proof}
 
\begin{lemma}  \label{lem5.4}
For any $n \in \Z^+$, \\
\vspace{-1.5em} $$\g C^n A, \, \, \g BC^n \in \freeLie.$$
\end{lemma}
\begin{proof}
From Proposition~\ref{prop5.3}, for all $n\in\N$, we have
\begin{eqnarray}
\nonumber q^n ((\ad B)\circ (\ad C)^n )(A)=(1-q)^n \{n\}_q \g C^n-(1-q)^n \{n+1\}_q C^{n+1}		
\end{eqnarray}
Let $L_2:=[q^n ((\ad B)\circ(\ad C)^n )(A),A]$. Now, observe
\begin{eqnarray}
\nonumber L_2 &=& [(1-q)^n \{n\}_q \g C^n-(1-q)^n \{n+1\}_q C^{n+1},A] \label{rel5.14} \\
\nonumber &=&  (1-q)^n\{n\}_q \g C^n A-(1-q)^n\{n+1\}_q C^{n+1} A \\
\nonumber & &-(1-q)^n\{n\}_q A\g C^n + (1-q)^n\{n+1\}_q AC^{n+1}, \\
\nonumber &=&  (1-q)^n\{n\}_q \g C^n A-(1-q)^n\{n+1\}_q C^{n+1} A \\
\nonumber & &-(1-q)^nq^n \{n\}_q \g C^n A+(1-q)^nq^{n+1} \{n+1\}_q C^{n+1} A, \\
\nonumber &=&  (1-q)^n(1- q^n ) \{n\}_q \g C^n A-(1-q)^n(1-q^{n+1} ) \{n+1\}_q C^{n+1} A , \\
\nonumber &=& (1-q)^{n+1} \{n\}_q \{n\}_q \g C^n A-(1-q)^{n+1} \{n+1\}_q \{n+1\}_q C^{n+1} A.  \label{rel5.15}
\end{eqnarray}
Hence, we have
\begin{eqnarray} \label{equaf5.17}
\nonumber(1-q)^{n+1} \{n\}_q \{n\}_q \g C^n A &=&(1-q)^{n+1} \{n+1\}_q \{n+1\}_q C^{n+1} A+L_2. 
\end{eqnarray}
which implies that $\g C^n A \in \freeLie$. 

It can be shown that $\g BC^n \in \freeLie$ in a similar manner.
\end{proof}
 We define elements $L_a, \, L_b \in \rq$ in the following manner:
 \begin{eqnarray}
 \nonumber L_a &=& q^n((-\ad A)^m \circ (\ad B)\circ (\ad C)^n )(A), \\
 \nonumber L_b &=& q^n ((-\ad B)^m \circ (\ad B)\circ (\ad C)^n )(A).
 \end{eqnarray}
\begin{proposition} \label{propof5.5} For any natural number $m$ and $n$,
\begin{eqnarray}
L_a &=& (1-q)^{n+m}\left( \{n\}_q^{m+1} \g C^n A^m- \{n+1\}_q^{m+1} C^{n+1} A^m \right) , \quad \emph{and} \label{eq5.16} \\
L_b&=& (1-q)^{n+m} \left(\{n+1\}_q^{m+1} B^m C^{n+1} - \{n\}_q^{m+1} \g B^m C^n\right). \label{eq5.17} 
\end{eqnarray}

\end{proposition}
\begin{proof}
We prove equations \eqref{eq5.16} and \eqref{eq5.17} by induction on $m$. It can be easily shown that \eqref{eq5.16} and \eqref{eq5.17} holds for $m=0$. We suppose equation \eqref{eq5.16}  holds for some $k \in \N$. 
Denote $q^n ((-\ad A)^{k+1} \circ (\ad B)\circ (\ad C)^n )(A)$ by $L_3$. \\
Now, observe that 
\begin{eqnarray}
\nonumber L_3 &=& (-\ad A)\circ(q^n ((-\ad A)^{k} \circ (\ad B)\circ (\ad C)^n )(A)), \\
\nonumber &=&[(1-q)^{n+k} \{n\}_q^{k+1} \g C^n A^k-(1-q)^{n+k} \{n+1\}_q^{k+1} C^{n+1} A^k, A], \\
\nonumber &=& (1-q)^{n+k} \{n\}_q^{k+1} \g C^n A^{k+1}-(1-q)^{n+k} \{n+1\}_q^{k+1} C^{n+1} A^{k+1} \\
\nonumber &\quad& -(1-q)^{n+k} \{n\}_q^{k+1} q^n \g C^n A^{k+1}+(1-q)^{n+k} \{n+1\}_q^{k+1}q^{n+1} C^{n+1} A^{k+1}, \\
\nonumber &=& (1-q^n)(1-q)^{n+k} \{n\}_q^{k+1} \g C^n A^{k+1} \\
\nonumber & &-(1-q^{n+1})(1-q)^{n+k} \{n+1\}_q^{k+1} C^{n+1} A^{k+1}, \\
\nonumber &=& (1-q)^{n+k+1} \{n\}_q^{k+2} \g C^n A^{k+1}-(1-q)^{n+k+1} \{n+1\}_q^{k+2} C^{n+1} A^{k+1}.
\end{eqnarray}
So  by induction, \eqref{eq5.16} holds for any $m \in \N$.

We can also do the same process for equation \eqref{eq5.17} which will consequently lead to the desired result.
\end{proof}

\begin{proposition} \label{propof5.6} For any $n, \, m \in \Z^+$, \\
\vspace{-1.5em}$$ \g  C^n A^m, \,  \,  \g B^m C^n \in \freeLie. $$
\end{proposition}
\begin{proof}
This follows immediately from Proposition~\ref{propof5.5}.
\end{proof}

The following proposition is a more general case of Proposition \eqref{propof5.6}.
\begin{proposition} \label{Propo5.7} 
 For any $m, \, n \in \Z^+$ and for any $h \in \N$, \\
\vspace{-1.5em}\begin{eqnarray} \label{ref5.7}
 \g^h C^n A^m, \,  \, \g^h B^mC^n \in \freeLie.
\end{eqnarray}
\end{proposition}
\begin{proof}
We prove this proposition by induction on $h$. We first consider $\g^h C^n A^m$. 

The case $h=0$ holds by Proposition~\ref{Propo5.2}. Also, it is clear that $h=1$ holds by the Lemma~\ref{lem5.4}. 

We assume that for some $k \in \N$, $\g^k C^n A^m$ is in $\freeLie$. We show that $\g^{k+1} C^n A^m$ is also in $\freeLie$. \\
Recall that $A$ and $C$ are in $\freeLie$, hence, $[\g^k C^n A^m,A]$ and $[\g^k C^n A^m,C]$ are also in $\freeLie$.

Observe that
\vspace{-0.5em}\begin{eqnarray}
 [\g^k C^n A^m, A]&=&(1-q^n) \g^k C^n A^{m+1}, \label{Libra11}
\end{eqnarray}
while
\vspace{-0.5em}\begin{eqnarray}
[\g^k C^n A^m, C]&=&(q^m\ -1)\g ^k C^{n+1} A^{m} \label{Libra13}
\end{eqnarray}
which imply that $\g^k C^{n} A^{m+1}, \ \g^k C^{n+1} A^{m}$ is in $\freeLie$. 

We  use $L_4$ to denote $(1-q)q^n[\g^k C^n A^{m+1},B]$. Notice that

\begin{eqnarray}
\nonumber L_4&=&(1-q)q^n \g^k C^n A^{m+1} B-(1-q)q^n B\g^k C^n A^{m+1}, \\
\nonumber&=&(1-q)q^n \g^k C^n A^m AB-(1-q)\g^k C^n BAA^m, \\
\nonumber&=&q^n \g^k C^n A^m (\g-qC)-\g^k C^n (\g-C)A^m, \\
\nonumber&=&q^n \g^{k+1} C^n A^m-q^{n+m+1} \g^k C^{n+1}  A^m-\g^{k+1} C^n A^m+\g^k C^{n+1} A^m, \\
\nonumber&=& (1-q^{n+m+1})\g^k C^{n+1} A^m-(1-q^n)\g^{k+1} C^n A^m, \\
&=&(1-q)\{n+m+1\}_q \g^k C^{n+1} A^m-(1-q)\{n\}_q \g^{k+1} C^n A^m.  \label{Libra12}
\end{eqnarray}
Consequently, we have  
\begin{eqnarray} \label{equaF5.28}
\nonumber \{n\}_q \g^{k+1} C^n A^m=\{n+m+1\}_q \g^k C^{n+1} A^m-q^n [\g^k C^n A^{m+1},B].
\end{eqnarray}
which implies that $ \g^{k+1} C^n A^m \in \freeLie$. By induction, the desired result follows.

We can show that $\g^h B^mC^n \in \freeLie$ for all natural number $h$ and positive integers $n, \ m$ in similar manner. 
\end{proof}

\begin{proposition} \label{Propo5.8}
Let $l, \, m, \, n \in \Z^+$ and $h \in \N$. If $m \neq l$,  then 
\begin{eqnarray}
 \g^h C^n A^m B^l, \,  \, \g^h C^n B^lA^m \in \freeLie.
\end{eqnarray}
Moreover, each of these elements are linear combination of elements  in \eqref{ref5.7}.
\end{proposition}
\begin{proof}
We consider cases where $m<l, \,  m>l$ for $\g^h C^n A^m B^l$ .

Suppose $m<l$. Then using Theorem~\ref{Theo4.3}, we have
\begin{eqnarray}
\nonumber (1-q)^m\g^h C^n A^m B^l &=& (1-q)^m\g^h C^n A^mB^m B^{l-m}, \\
\nonumber&=&\g^h C^n \left(\sum_{i=0}^{m}(-1)^i q^{\binom{i+1}{2}} {\binom{m}{i}}_q \g^{m-i} \LieAB^i\right) B^{l-m}, \\
\nonumber&=&\g^h C^n \left(\sum_{i=0}^{m}(-1)^i q^{\binom{i+1}{2}} {\binom{m}{i}}_q \g^{m-i} C^i\right) B^{l-m}, \\
&=& \sum_{i=0}^{m}(-1)^i q^{\binom{i+1}{2}} {\binom{m}{i}}_q \g^{h+m-i} C^{n+i}B^{l-m}. \label{sum3}
\end{eqnarray}
Using Proposition~\ref{Profin4.2}, we further simplify the right-hand side of \eqref{sum3} into the following expression.
\begin{eqnarray}
& & \sum_{i=0}^{m}(-1)^i q^{\binom{i+1}{2} +(l-m)(n+i)} {\binom{m}{i}}_q \g^{h+m-i}B^{l-m} C^{n+i}. \label{ref5.27}
\end{eqnarray}

Observe that each term of the summation in \eqref{ref5.27} is in $\freeLie$ by Proposition~\ref{Propo5.7} hence, it follows that $\g^h C^n A^m B^l \in \freeLie$ whenever $m<l$. 

Meanwhile, suppose $m>l$. Then using Theorem~\ref{Theo4.3} would consequently result to 

\begin{eqnarray}
\nonumber (1-q)^l\g^h C^n A^m B^l &=& (1-q)^l\g^h C^n A^{m-l}A^lB^l, \\
\nonumber&=& \g^h C^n A^{m-l}\left(\sum_{i=0}^{l}(-1)^i q^{\binom{i+1}{2}} {\binom{l}{i}}_q \g^{l-i} \LieAB^i\right), \\
\nonumber&=& \g^h C^n A^{m-l}\left(\sum_{i=0}^{l}(-1)^i q^{\binom{i+1}{2}} {\binom{l}{i}}_q \g^{l-i} C^i\right), \\
\nonumber&=& \g^h C^n \sum_{i=0}^{l}(-1)^i q^{\binom{i+1}{2}+(im-il)} {\binom{l}{i}}_q \g^{l-i} C^iA^{m-l}, \\
&=& \sum_{i=0}^{l}(-1)^i q^{\binom{i+1}{2}+(im-il)} {\binom{l}{i}}_q \g^{h+l-i} C^{n+i}A^{m-l}.  \label{ref5.29}
\end{eqnarray}
Observe that each term of the summation in \eqref{ref5.29} are in $\freeLie$ by Proposition~\ref{Propo5.7} hence, it follows that $\g^h C^n A^m B^l \in \freeLie$ whenever $m>l$.

We can show that $\g^h C^n B^lA^m$ are also elements of $\freeLie$ in similar manner.
\end{proof}

\begin{theorem} \label{Theof5.10} For any $m, \, n, \, h \in \Z^+$ and for any $j,k \in \N$, the element 
\begin{eqnarray}
\nonumber (1-q)^{m-1}q^{\binom{m}{2}}q^{m(n+h)}[\g^kB^mC^n, \g^jC^h A^m] 
\end{eqnarray}
of $\rq$ is equal to
\begin{eqnarray}
\nonumber & & \sum_{i=0}^{m}(-1)^i q^{\binom{m-i}{2}}\{mn+mh+im\}_q {\binom{m}{i}}_q \g^{m-i+j+k}C^{i+h+n}. 
\end{eqnarray}
\end{theorem}
\begin{proof}
Observe that 

\begin{eqnarray}
\nonumber q^{\binom{m}{2}}(1-q)^mq^{m(n+h)}[\g^kB^mC^n, \g^jC^h A^m] &=& q^{\binom{m}{2}}(1-q)^mq^{m(n+h)}\g^{k+j}B^mC^{n+h} A^m\\
\nonumber & & - q^{\binom{m}{2}}(1-q)^mq^{m(n+h)}\g^{j+k}C^h A^mB^mC^n, \\
\nonumber &=& q^{\binom{m}{2}}(1-q)^m\g^{k+j}B^m A^mC^{n+h}\\
\nonumber & & -  q^{\binom{m}{2}}(1-q)^mq^{m(n+h)}\g^{j+k} A^mB^mC^{h+n}, \\
\nonumber &=& (q^{\binom{m}{2}}(1-q)^mB^m A^m\\
 \nonumber & & -  q^{\binom{m}{2}}(1-q)^m q^{m(n+h)} A^mB^m)\g^{j+k}C^{h+n}. \label{equaf5.39}
\end{eqnarray}

From Theorem~\ref{Theo4.3}, we have 
 \begin{eqnarray}
\nonumber q^{\binom{m}{2}}q^{m(n+h)}(1-q)^mA^m B^m  &=& \sum_{i=0}^{m}{(-1)^i q^{\binom{i+1}{2}+{\frac{m^2 -m +2mn+2mh}{2}}} {\binom{m}{i}}_q \g^{m-i} C^i},   \label{eqf67} \\
 \nonumber q^{\binom{m}{2}}(1-q)^mB^m A^m &=& \sum_{i=0}^{m}{(-1)^i q^{\binom{m-i}{2}} {\binom{m}{i}}_q \g^{m-i}C^i}.  \label{eqf68}
 \end{eqnarray} 
 Also, by some computation, we can show that \\
 \begin{eqnarray}
\nonumber \left(\binom{i+1}{2}+\frac{m^2 -m +2mn+2mh}{2}\right) - \binom{m-i}{2} &=& mn+mh+im. 
\end{eqnarray}
Using these equations to solve for the equivalent expression of \eqref{equaf5.39} would consequently lead to the following.
 \begin{eqnarray}
& & \sum_{i=0}^{m}(-1)^i q^{\binom{m-i}{2}} (1-q)\{mn+mh+im\}_q {\binom{m}{i}}_q \g^{m-i+j+k}C^{i+h+n}. 
 \end{eqnarray}
from which the desired result follows.
\end{proof}

  \begin{proposition}\label{Fpropo5.11} Fix a natural number $r$ and positive integers $n$ and $h$. For any positive integer $m >r$,
\begin{eqnarray} \label{rel272}
 \sum_{i=0}^{m}(-1)^i q^{\binom{m-i}{2}}(q^{n+h+i})^r \binom{m}{i}_q &=& 0. 
\end{eqnarray}
  \end{proposition}
\begin{proof}
We prove this by induction on $m-r$.
If $m-r=1$, then $m=r+1$ which gives us the  case 
\begin{eqnarray} \label{eqlemzero}
\sum_{i=0}^{r+1}(-1)^i q^{\binom{r+1-i}{2}}(q^{n+h+i})^r \binom{r+1}{i}_q.
\end{eqnarray}
We have to show that \eqref{eqlemzero} is indeed equal  to zero. We do this by induction on $r$. The case $r=0$ is a trivial case.

Suppose \eqref{eqlemzero} holds for some $k \in \N$, that is, \begin{eqnarray}
\sum_{i=0}^{k+1}(-1)^i q^{\binom{k+1-i}{2}}(q^{n+h+i})^k \binom{k+1}{i}_q &=&0.
\end{eqnarray}
We prove that it is also true for $k+1$. We use $L_5$ to denote the following expression.
$$\sum_{i=0}^{k+2}(-1)^i q^{\binom{k+2-i}{2}}(q^{n+h+i})^{k+1} \binom{k+2}{i}_q$$
\\Observe that using equation \eqref{eq272}, we have 
\begin{eqnarray}
 L_5 &=& \sum_{i=0}^{k+2}(-1)^i q^{\binom{k+2-i}{2}}(q^{n+h+i})^{k+1} \left(\binom{k+1}{i-1}_q + q^i\binom{k+1}{i}_q \right). \label{sum4}
\end{eqnarray}
The right-hand expression of \eqref{sum4}
is equivalent to the following.
\begin{eqnarray}
 \nonumber& &\sum_{i=0}^{k+2}(-1)^i q^{\binom{k+2-i}{2}}q^{n+h+i}(q^{n+h+i})^k\binom{k+1}{i-1}_q \\ 
 & & +\sum_{i=0}^{k+2}(-1)^i q^{\binom{k+2-i}{2}}q^iq^{n+h+i}(q^{n+h+i})^k \binom{k+1}{i}_q. \label{sum5}
\end{eqnarray}
By some routine computations, we can simplify expression \eqref{sum5} such that
\begin{eqnarray}
\nonumber L_5  &=&  -q^{n+h+k+1}\sum_{i=-1}^{k+1}(-1)^{i} q^{\binom{k+1-i}{2}}q^{i}(q^{n+h+i})^k\binom{k+1}{i}_q \\
\nonumber&\quad& +q^{n+h+k+1}\sum_{i=0}^{k+2}(-1)^i q^{\binom{k+1-i}{2}}q^{i}(q^{n+h+i})^k \binom{k+1}{i}_q 
\end{eqnarray}
the second summation in which can be decomposed such that we further have
\begin{eqnarray}
\nonumber L_5  &=&  -q^{n+h+k+1}\sum_{i=0}^{k+1}(-1)^{i} q^{\binom{k+1-i}{2}}q^{i}(q^{n+h+i})^k\binom{k+1}{i}_q \\
\nonumber & & -q^{n+h+k+1}(-1)^{-1} q^{\binom{k+2}{2}}q^{-1}(q^{n+h-1})^k\binom{k+1}{-1}_q   \\
\nonumber&\quad& +q^{n+h+k+1}\sum_{i=0}^{k+1}(-1)^i q^{\binom{k+1-i}{2}}q^{i}(q^{n+h+i})^k \binom{k+1}{i}_q \\
\nonumber & & +q^{n+h+k+1}(-1)^{k+2} q^{\binom{1}{2}}q^{k+2}(q^{n+h+k+2})^k \binom{k+1}{k+2}_q. 
\end{eqnarray}
By the interpretation of \eqref{eq13} for the case when $p<0$ or $p>n$, we would consequently have 
\begin{eqnarray}
\nonumber L_5 &=&  (q^{n+h+k+1}-q^{n+h+k+1})\sum_{i=0}^{k+1}(-1)^{i} q^{\binom{k+1-i}{2}}q^{i}(q^{n+h+i})^r\binom{k+1}{i}_q =0.
\end{eqnarray} \\
which implies that \eqref{eqlemzero} is indeed equal to zero, and this implies that \eqref{rel272} holds when $m-r=1$.  

Now, we assume \eqref{rel272} is true for some $m-r=s \in \N$ . Thus we have $m=r+s$. This means that we assume that 
\begin{eqnarray}
 \sum_{i=0}^{r+s}(-1)^i q^{\binom{r+s-i}{2}}(q^{n+h+i})^r \binom{r+s}{i}_q &=&0,\label{induct}
\end{eqnarray}
holds for some $s \in \N$. We show that \eqref{induct} is also true when $m-r=s+1$, that is, $m=r+s+1$.
We denote the expression 
$$\sum_{i=0}^{r+s+1}(-1)^i q^{\binom{r+s+1-i}{2}}(q^{n+h+i})^r \binom{r+s+1}{i}_q,$$
by $L_6$. Observe that using \eqref{eq272}, we have
\begin{eqnarray}
\nonumber L_6 &=& \sum_{i=0}^{r+s+1}(-1)^i q^{\binom{r+s+1-i}{2}}(q^{n+h+i})^r \left(\binom{r+s}{i-1}_q + q^i\binom{r+s}{i}_q \right).
\end{eqnarray}

By some routine computations, we have
\begin{eqnarray}
\nonumber L_6 &=&  \sum_{i=0}^{r+s+1}(-1)^i q^{\binom{r+s+1-i}{2}}(q^{n+h+i})^r\binom{r+s}{i-1}_q \\
\nonumber & & +q^{r+s}\sum_{i=0}^{r+s+1}(-1)^i q^{\binom{r+s-i}{2}}(q^{n+h+i})^r \binom{r+s}{i}_q
\end{eqnarray} 
in which the right-hand side can be decomposed in a way such that
\begin{eqnarray}
\nonumber L_6 &=&  \sum_{i=0}^{r+s+1}(-1)^i q^{\binom{r+s+1-i}{2}}(q^{n+h+i})^r\binom{r+s}{i-1}_q \\
\nonumber & & +q^{r+s}\sum_{i=0}^{r+s}(-1)^i q^{\binom{r+s-i}{2}}(q^{n+h+i})^r \binom{r+s}{i}_q \\
\nonumber&\quad& + (-1)^{r+s+1} q^{\binom{r+s+1-(r+s+1)}{2}}(q^{n+h+r+s+1})^r q^{r+s+1}\binom{r+s}{r+s+1}_q.
\end{eqnarray}
Notice that by the inductive hypothesis, in addition to the interpretation of \eqref{eq13} whenever $p>n$, the above equation would simply be reduced into 
\begin{eqnarray}
 L_6 &=&  -q^r\sum_{i=-1}^{r+s}(-1)^{i} q^{\binom{r+s-i}{2}}(q^{n+h+i})^r\binom{r+s}{i}_q. \label{sum6}
\end{eqnarray}
The equation \eqref{sum6} can be further reduced into 
\begin{eqnarray}
 \nonumber L_6 &=&  -q^r\sum_{i=0}^{r+s}(-1)^{i} q^{\binom{r+s-i}{2}}(q^{n+h+i})^r\binom{r+s}{i}_q \\
 \nonumber & & +(-q^r)(-1)^{-1} q^{\binom{r+s+1}{2}}(q^{n+h-1})^r\binom{r+s}{-1}_q \end{eqnarray}
which, by using the inductive hypothesis, implies that $L_6=0$. This shows that \eqref{rel272} also holds for $s+1$.
 
 Thus, by induction on $m-r$, the result follows.  
\end{proof}

Some elements of our candidate basis for $\freeLie$ includes elements of the form 
\begin{eqnarray}
\g^h C^n A^m, \ \ \g^h B^l C^n\label{vectorgoal}
\end{eqnarray}
where $l,m,n\in\Z^+$ and $h\in\N$. Our goal is to show that the Lie bracket of any two vectors among \eqref{vectorgoal} is in $\freeLie$. We consider at this point the case $m=h$. That is, we show that $[\g^h C^n A^m, \g^j B^m C^k]$ is a  linear combination of elements of the form $\g^r C^s$ where $r, \, s \in \N$. More specifically, we show that $[\g^h C^n A^m,\g^j B^m C^k]$ is a linear combination of elements of the form  
 \begin{equation} \label{basisf1}
     \{n+1\}_q \g^h C^{n+1}-\{n\}_q \g^{h+1} C^n,
 \end{equation}
which is equal to $q^n [\g^h C^n A,B] \in \freeLie$ using previously established reordering formulae. In other words, our candidate basis should include vectors of the form \eqref{basisf1}. 

We need the following relations and computations in order to prove that  $[\g^h C^n A^m, \g^j B^m C^k]$ is a linear combination of \eqref{basisf1}. Given $i\in\Z^+$, define
\begin{eqnarray}
\nonumber \psi_i &=&\{i+h+n\}_q \g^{m-i+j+k} C^{i+h+n}-\{i+h+n-1\}_q \g^{m-i+j+k+1} C^{i+h+n-1} .
\end{eqnarray}
Given $m\in\Z^+$, we are interested in properties of $\psi_i$ where $i\in\{1,2,\ldots,m\}$. We show below a sequence of relations that has some behavior related to ``telescoping series" computations, and this gives the intuition for what we want to accomplish.
\begin{eqnarray}
\nonumber \{1+h+n\}_q \g^{m-1+j+k} C^{1+h+n}&=& \psi_1 +\{h+n\}_q \g^{m+j+k} C^{h+n} \label{rela281} \\
\nonumber \{2+h+n\}_q \g^{m-2+j+k} C^{2+h+n} &=&\psi_2 + \psi_1 +\{h+n\}_q \g^{m+j+k} C^{h+n}  \\
\nonumber \{3+h+n\}_q \g^{m-3+j+k} C^{3+h+n}&=&\psi_3+\psi_2 + \psi_1 +\{h+n\}_q \g^{m+j+k} C^{h+n}  \\
\nonumber &\vdots& \\
\nonumber \{(m-1)+h+n\}_q \g^{1+j+k} C^{(m-1)+h+n}&=& \sum_{i=1}^{m-1}\psi_i +\{h+n\}_q \g^{m+j+k} C^{h+n} \\
\{m+h+n\}_q \g^{j+k} C^{m+h+n}&=& \sum_{i=1}^{m}\psi_i +\{h+n\}_q \g^{m+j+k} C^{h+n}  \label{rela285}
\end{eqnarray}
The above relations can be proven by routine computations and arguments, and they are instrumental in the proofs of our succeeding results.

 \begin{theorem} \label{theof5.12}
 For any $m, \, n, \, h \in \Z^+$ and for any $j, \, k \in \N$,
\begin{eqnarray}
 L_c &:=& \sum_{i=0}^{m}(-1)^i q^{\binom{m-i}{2}}\{m\}_{q^{n+h+i}} {\binom{m}{i}}_q \sum_{t=0}^{i}\psi_t   \label{close1}
\end{eqnarray}  
\vspace{-1em}where $\psi_t= \{t+h+n\}_q \g^{j+k} C^{t+h+n}-\{(t-1)+h+n\}_q \g^{1+j+k} C^{(t-1)+h+n}$ 
 \end{theorem}
\vspace{0.7em}\begin{proof}
Recall that from Theorem~\ref{Theof5.10}, \\
\vspace{-1em}\begin{eqnarray} \label{equa282}
L_c&=& \sum_{i=0}^{m}(-1)^i q^{\binom{m-i}{2}}\{mn+mh+im\}_q {\binom{m}{i}}_q \g^{m-i+j+k}C^{i+h+n} 
\end{eqnarray}
Using \eqref{eqf12}, we can simplify the right-hand side expression of \eqref{equa282} into the following:
\begin{eqnarray}
& & \sum_{i=0}^{m}(-1)^i q^{\binom{m-i}{2}}\{m\}_{q^{n+h+i}}\{n+h+i\}_q {\binom{m}{i}}_q \g^{m-i+j+k}C^{i+h+n} \label{sum7}
\end{eqnarray}
Expanding the summation \eqref{sum7}, we have
\vspace{-1em}\begin{eqnarray}
\nonumber L_c &=& q^{\binom{m}{2}}\{m\}_{q^{n+h}}\{n+h\}_q {\binom{m}{0}}_q \g^{m+j+k}C^{h+n} \\
\nonumber & & -q^{\binom{m-1}{2}}\{m\}_{q^{n+h+1}}\{n+h+1\}_q {\binom{m}{1}}_q \g^{m-1+j+k}C^{1+h+n} \\
\nonumber & & +q^{\binom{m-2}{2}}\{m\}_{q^{n+h+2}}\{n+h+2\}_q {\binom{m}{2}}_q \g^{m-2+j+k}C^{2+h+n} \\
\nonumber & &\vdots \\
\nonumber & & +(-1)^{m-1} q^{\binom{1)}{2}}\{m\}_{q^{n+h+(m-1)}}\{n+h+(m-1)\}_q {\binom{m}{m-1}}_q \g^{1+j+k}C^{m-1+h+n} \\
\nonumber & &+(-1)^m q^{\binom{0}{2}}\{m\}_{q^{n+h+m}}\{n+h+m\}_q {\binom{m}{m}}_q \g^{j+k}C^{m+h+n}.
\end{eqnarray}
Using equations described in \eqref{rela285} to simplify the above expansion would give us
\begin{eqnarray}
\nonumber L_c &=& q^{\binom{m}{2}}\{m\}_{q^{n+h}}\{n+h\}_q {\binom{m}{0}}_q \g^{m+j+k}C^{h+n} \\
\nonumber & & -q^{\binom{m-1}{2}}\{m\}_{q^{n+h+1}} {\binom{m}{1}}_q \left(\psi_1 +\{h+n\}_q \g^{m+j+k} C^{h+n}\right) \\
\nonumber & & +q^{\binom{m-2}{2}}\{m\}_{q^{n+h+2}} {\binom{m}{2}}_q \left(\psi_2 + \psi_1 +\{h+n\}_q \g^{m+j+k} C^{h+n}\right) \\
\nonumber & & \vdots \\
\nonumber & & +(-1)^{m-1} q^{\binom{1)}{2}}\{m\}_{q^{n+h+(m-1)}} {\binom{m}{m-1}}_q \left(\sum_{i=1}^{m-1}\psi_i +\{h+n\}_q \g^{m+j+k} C^{h+n} \right)  \\
\nonumber & &+(-1)^m q^{\binom{0}{2}}\{m\}_{q^{n+h+m}} {\binom{m}{m}}_q \left(\sum_{i=1}^{m}\psi_i +\{h+n\}_q \g^{m+j+k} C^{h+n}\right)
\end{eqnarray}
Hence, 
\vspace{-1em}\begin{eqnarray}
\nonumber L_c &=& \sum_{i=1}^{m}(-1)^iq^{\binom{m-i}{2}}\{m\}_{q^{n+h+i}} \binom{m}{i}\sum_{t=0}^{i}\psi_t \\
\nonumber & & +\sum_{i=0}^{m}(-1)^iq^{\binom{m-i}{2}}\{m\}_{q^{n+h+i}}\{h+n\}_q\binom{m}{i}\g^{m+j+k} C^{h+n}.
\end{eqnarray}
Let $L_7=\sum_{i=0}^{m}(-1)^iq^{\binom{m-i}{2}}\{m\}_{q^{n+h+i}}\binom{m}{i}.$ Then we have 
\begin{eqnarray}
\nonumber L_c &=& \sum_{i=1}^{m}(-1)^iq^{\binom{m-i}{2}}\{m\}_{q^{n+h+i}} \binom{m}{i}\sum_{t=0}^{i}\psi_t +\{h+n\}_q\g^{m+j+k} C^{h+n}L_7.
\end{eqnarray}
Now, observe that 
\vspace{-0.8em}\begin{eqnarray}
\nonumber L_7 &=&\sum_{i=0}^{m}(-1)^iq^{\binom{m-i}{2}}(1+q^{n+h+i}+(q^{n+h+i})^2+\ldots +(q^{n+h+i})^{m-1})\binom{m}{i}, \\
\nonumber &=&\sum_{i=0}^{m}(-1)^iq^{\binom{m-i}{2}}\binom{m}{i}+\sum_{i=0}^{m}(-1)^iq^{\binom{m-i}{2}}q^{n+h+i}\binom{m}{i} \\
\nonumber & &+\sum_{i=0}^{m}(-1)^iq^{\binom{m-i}{2}}(q^{n+h+i})^2\binom{m}{i}\\
\nonumber &\quad& +\ldots +\sum_{i=0}^{m}(-1)^iq^{\binom{m-i}{2}}(q^{n+h+i})^{m-1}\binom{m}{i}.
\end{eqnarray}
So by Proposition~\ref{Fpropo5.11}, we simply have $L_7=0$.
Consequently, 
\vspace{-1em}\begin{eqnarray}
\nonumber L_c &=&\sum_{i=1}^{m}(-1)^iq^{\binom{m-i}{2}}\{m\}_{q^{n+h+i}} \binom{m}{i}\sum_{t=0}^{i}\psi_t 
+ \{h+n\}_q\g^{m+j+k} C^{h+n}(0), \\
\nonumber &=&\sum_{i=1}^{m}(-1)^iq^{\binom{m-i}{2}}\{m\}_{q^{n+h+i}} \binom{m}{i}\sum_{t=0}^{i}\psi_t 
\end{eqnarray}
and we are done.
\end{proof}

Notice the following equations that show the Lie bracket of pairs of vectors taken from our candidate basis as Lie polynomials in $A,B$. These are obtained by using Proposition~\ref{Profin4.2} to perform some reordering, and also by using some relations related to $\nqcombi$ to simplify scalar coefficients. 
\begin{eqnarray}
q^n\left[\g^h B^l C^n,A\right] &=& \{n\}_q\g^{h+1}B^{l-1} C^n\nonumber\\
& &-\{n+l\}_q\g^hB^{l-1} C^{n+1}, \label{Libra1}\\
\left[\g^h B^l C^n,B\right] &=& (q^n-1)\g^h B^{l+1} C^n, \label{Libra2}\\
\left[\g^h B^l C^n,C\right]&=&(1-q^l)\g^h B^l C^{n+1}, \label{Libra3}\\
\left[\g^h C^n A^m,\g^j C^h A^l\right]&=& (q^{mh}-q^{ln})\g^{h+j} C^{n+h}A^{m+l}, \label{Libra4}\\ 
\left[\g^h B^l C^n,\g^j B^m C^h\right]&=& (q^{2mn}- q^{2lh})\g^{h+j} B^{l+m} C^{h+n}, \label{Libra5} \\
\nonumber q^{l(k+n)}\left[\g^h C^n A^m, \g^j B^l C^k\right] &=&q^{ln+km} \g^{h+j} C^{n+k}A^mB^l \\
& &-\g^{h+j} C^{k+n}B^lA^m, \label{Libra6} \\
\nonumber \left[\{n+1\}_q \g^h C^{n+1}-\{n\}_q \g^{h+1} C^n,A\right]&=&(1-q^{n+1} )\{n+1\}_q \g^h C^{n+1} A \\
 & & -(1-q^n)\{n\}_q \g^{h+1} C^n A, \label{Libra7} \\
\nonumber \left[\{n+1\}_q \g^h C^{n+1}-\{n\}_q \g^{h+1} C^n,B\right]&=& (q^{n+1}-1)\{n+1\}_q \g^h BC^{n+1} \\
 & &-(q^n-1)\{n\}_q \g^{h+1} BC^n, \label{Libra8} \\
  \left[\g^h C^n A^m, \{n+1\}_q \g^k C^{n+1}-\{n\}_q \g^{k+1} C^n \right]&=& (1- q^{mn})\{n\}_q\g^{h+k+1} C^{2n} A^m \label{Libra9}\\ & & -(1-q^{m(n+1)})\{n+1\}_q\g^{h+k} C^{2n+1}A^m,\nonumber  \\
 \nonumber \left[\g^h B^l C^n, \{m+1\}_q \g^j C^{m+1}-\{m\}_q \g^{j+1} C^m \right]&=&(1-q^{l(m+1)})\{m+1\}_q \g^{j+h}B^l C^{n+m+1} \\
 & &-(1-q^{lm})\{m\}_q \g^{j+1+h} B^lC^{m+1}. \label{Libra10}
\end{eqnarray}

Clearly, the right-hand sides of \eqref{Libra1} to \eqref{Libra10} are in the Lie subalgebra $\freeLie$ by Proposition~\ref{Propo5.7} and \ref{Propo5.8}.

We now exhibit a basis for the Lie subalgebra $\freeLie$ of $\rq$ generated by $A, \, B$ 

\begin{theorem} The following elements form a basis for $\freeLie$: 
\begin{eqnarray} 
& & \g^h C^n A^m, \g^h B^m C^n, \nonumber\\
& & \{n+1\}_q \g^h C^{n+1}-\{n\}_q \g^{h+1} C^n,\nonumber\\
&  & A, B, C,\quad\quad\quad\quad\quad\quad\quad\quad\quad\quad\quad\quad(h \in \N,\  m,n \in \Z^+.)\label{ref5.52}
\end{eqnarray}
\end{theorem}
\begin{proof}
Let $\scK$ be the span of \eqref{ref5.52}. To show that $\scK$ is equal to $\freeLie$, we only have to show the following conditions are satisfied:
\begin{enumerate}[(i)]
    \item $A, B \in \scK$,
    \item $\scK$ is a Lie subalgebra of $\rq$
    \item $\scK \subseteq \freeLie$.
\end{enumerate}
 The condition (i) immediately follows from the definition of $\scK$. For condition (ii), we show that any basis elements $L, R \in$ of $ \scK$, $[L, R]$ is a linear combination of \eqref{ref5.52}. Let $k, l, m, n \in \Z^+$ and $j, k \in \N$.
 Define $\beta_1, \beta_2 \in \scK$ in the following manner:
 \begin{eqnarray}
  \beta_1 &=& \{n+1\}_q \g^j C^{n+1}-\{n\}_q \g^{j+1} C^n, \\ \beta_2 &=& \{m+1\}_q \g^j C^{m+1}-\{m\}_q \g^{j+1} C^m.
 \end{eqnarray}
  We summarize all the cases for $[L, R]$ in the following table.
\begin{center}
\begin{tabular}{ |c||c|c|c|c|c|c| } 
 \hline
 $[ \ ,\ ]$ & $A$ & $B$ & $C$ & $\g^j C^k A^l$ & $\g^j B^l C^k$ &      $\beta_2$  \\ 
 \hline\hline
  $A$ & $0$ &  & &  &  &  \\ 
 \hline 
 $B$ & $-C$ & $0$ & &  &  &  \\ 
 \hline 
 $C$ & $(1-q)CA$ & $(q-1)BC$ & $0$ &  &  &  \\ 
 \hline 
 $\g^k C^n A^m$  & \eqref{Libra11} & \eqref{Libra12} &\eqref{Libra13} & \eqref{Libra4} &  &  \\ 
 \hline 
 $\g^k B^m C^n$ & \eqref{Libra1} & \eqref{Libra2} & \eqref{Libra3} & \eqref{Libra6}, \eqref{close1} & \eqref{Libra5} &  \\ 
 \hline
 $\beta_1$ & \eqref{Libra7} & \eqref{Libra8} & $0$ & \eqref{Libra9} & \eqref{Libra10} & $0$ \\
\hline 
\end{tabular}
\end{center}
 Notice that all possibilities for $L$ are listed in the  first column while all possibilities for $R$ are in the first row. For cell entries which are either $0$ or elements of $\rq$, condition (ii) is clearly satisfied. As for cells with equation numbers as entries, the pertinent equations show how $[L,R]$ is a linear combination of \eqref{ref5.52}. Furthermore, we left empty the  cells above the main diagonal because  of the skew-symmetric property of the Lie bracket. Hence, condition (ii) clearly  follows from the table. 
 
The condition (iii) is clearly satisfied by $A, B, C$.
Also, by Proposition~\ref{Propo5.7}, $\g^h C^n A^m , \, \g^h B^l C^n \in \freeLie$.
Furthermore, $$\{n+1\}_q \g^h C^{n+1}-\{n\}_q \g^{h+1} C^n=q^n [\g^h C^n A,B] \in \freeLie.$$ Thus, we have shown that each element in \eqref{ref5.52} is in the Lie subalgebra $\freeLie$ thereby satisfying condition (iii). 
This completes the proof.
\end{proof}

\subsection*{Acknowledgements}
The authors want to express their gratitude for the support of the first author's colleagues and the second author's mentors from the Mathematics and Statistics Department of the College of Science of De La Salle University, Manila. While this collaboration was in preparation, the first author was supported by the Commission for Developing Countries of the International Mathematical Union (IMU-CDC), and by the Mathematics and Applied Mathematics (MAM) Research Environment of M\"alardalen University, V\"aster\aa s, Sweden, while the second author was supported by the Science Education Institute of the Department of Science and Technology (DOST-SEI) of the Republic of the Philippines.

\end{document}